\documentclass[a4paper,reqno]{amsart} 

\usepackage{amssymb}
\usepackage[mathcal]{euscript}

\usepackage{color}
\newcommand{\rojo}{}

\usepackage{tikz-cd} 

\usepackage{hyperref}  

\newcommand{\bydef}{:=}

\newcommand{\Skew}{\mathrm{Skew}}
\newcommand{\id}{\mathrm{id}}

\newcommand{\trace}{\mathrm{tr}}

\newcommand{\cA}{\mathcal{A}}
\newcommand{\cB}{\mathcal{B}} 
\newcommand{\cC}{\mathcal{C}}
\newcommand{\cD}{\mathcal{D}} 

\newcommand{\cH}{\mathcal{H}} 

\newcommand{\cJ}{\mathcal{J}} 

\newcommand{\cL}{\mathcal{L}}

\newcommand{\cS}{\mathcal{S}} 
\newcommand{\cT}{\mathcal{T}}

\newcommand{\cW}{\mathcal{W}}

\newcommand{\cJord}{{\mathcal{J}ord}}

\newcommand{\frs}{{\mathfrak s}}

\newcommand{\ZZ}{\mathbb{Z}}

\newcommand{\FF}{\mathbb{F}} 

\DeclareMathOperator{\Hom}{\mathrm{Hom}}
\DeclareMathOperator{\End}{\mathrm{End}}

\DeclareMathOperator{\AAut}{\mathbf{Aut}}
\DeclareMathOperator{\Der}{\mathrm{Der}}

\DeclareMathOperator{\Mat}{\mathrm{Mat}}

\DeclareMathOperator{\Instrl}{\mathrm{Instrl}}

\newcommand{\ad}{\mathrm{ad}}

\newcommand{\frsl}{{\mathfrak{sl}}}
\newcommand{\frsp}{{\mathfrak{sp}}}
\newcommand{\frso}{{\mathfrak{so}}}

\newcommand{\frgl}{{\mathfrak{gl}}}

\newcommand{\SLs}{\mathbf{SL}}

\providecommand{\espan}[1]{\operatorname{span}\left\{ #1\right\}}

\providecommand{\ptriple}[1]{\boldsymbol{(}#1\boldsymbol{)}}

\newenvironment{romanenumerate} 
{\begin{enumerate}

}{\end{enumerate}}

\newcommand{\Cl}{\mathfrak{Cl}} 

\newcommand{\bup}{\textup{b}} 

\newcommand{\nup}{\textup{n}}

\newcommand{\hup}{\textup{h}}

\newtheorem{theorem}{Theorem}[section]
\newtheorem{proposition}[theorem]{Proposition}
\newtheorem{lemma}[theorem]{Lemma}

\theoremstyle{definition} 
\newtheorem{definition}[theorem]{Definition}

\theoremstyle{remark} \newtheorem{remark}[theorem]{Remark}
\numberwithin{equation}{section}

\begin{document}

\title{Short $(\SLs_2\times\SLs_2)$-structures on Lie algebras}

\author[P.D. Beites]{Patricia D. Beites}
\address{Departamento de
Matem\'{a}tica e Centro de Matem\'atica e 
Aplica\c{c}\~oes 
da Universidade da Beira Interior,
Universidade da Beira Interior, 6201-001 Covilh\~{a}, Portugal}
\email{pbeites@ubi.pt}

\author[A.S.~C\'ordova-Mart{\'\i}nez]{Alejandra S.~C\'ordova-Mart{\'\i}nez}
\address{Departamento de
Matem\'{a}ticas e Instituto Universitario de Matem\'aticas y
Aplicaciones, Universidad de Zaragoza, 50009 Zaragoza, Spain}
\email{acordova@unizar.es} 
\thanks{A.S.~C\'ordova-Mart{\'\i}nez and A.~Elduque were supported by grant PID2021-123461NB-C21, funded by 
MCIN/AEI/10.13039/501100011033 and by
 ``ERDF A way of making Europe''. A.S.~C\'ordova-Mart{\'\i}nez was also supported
 by grant S60\_20R (Gobierno de 
Arag\'on, Grupo de investigaci\'on ``Investigaci\'on en Educaci\'on Matem\'atica''); and A.~Elduque by
grant E22\_20R (Gobierno de 
Arag\'on, Grupo de investigaci\'on ``\'Algebra y Geometr{\'\i}a'').}

\author[I.~Cunha]{Isabel Cunha}
\address{Departamento de
Matem\'{a}tica e Centro de Matem\'atica e Aplica\c{c}\~oes da Universidade da Beira Interior,
Universidade da Beira Interior, 6201-001 Covilh\~{a}, Portugal}
\email{icunha@ubi.pt} 
\thanks{P.D.~Beites and I.~Cunha were supported by FCT (Funda\c{c}\~ao para a Ci\^encia e
a Tecnologia, Portugal), research project UIDB/00212/2020 of CMA-UBI (Centro de 
Matem\'atica e Aplica\c{c}\~oes, Universidade da Beira Interior, Portugal). P.D.~Beites was
also supported by grant PID2021-123461NB-C22, funded by 
MCIN/AEI/10.13039/501100011033 and by
 ``ERDF A way of making Europe'' (Spain).}

\author[A.~Elduque]{Alberto Elduque} 
\address{Departamento de
Matem\'{a}ticas e Instituto Universitario de Matem\'aticas y
Aplicaciones, Universidad de Zaragoza, 50009 Zaragoza, Spain}
\email{elduque@unizar.es}




\begin{abstract}
$\mathbf{S}$-structures on Lie algebras, introduced by Vinberg, represent a broad generalization of the
notion of gradings by abelian groups. Gradings by, not necessarily reduced, root systems provide many
examples of natural $\mathbf{S}$-structures. Here we deal with a situation not covered by 
these gradings: the short $(\SLs_2\times\SLs_2)$-structures, where the reductive group is the simplest 
semisimple but not simple reductive group. The algebraic objects that coordinatize these 
structures are the $J$-ternary algebras of Allison, endowed with a nontrivial idempotent.
\end{abstract}

\maketitle

\begin{center}
\textit{In memory of Georgia Benkart}
\end{center}

\bigskip

\section{Introduction}\label{se:intro}

\emph{All the algebras considered will be defined over an arbitrary ground field $\FF$ of
characteristic $\neq 2,3$. Tensor products over $\FF$ will simply be written as $\otimes$, instead
of $\otimes_\FF$. Algebraic groups over $\FF$ will be understood in the sense of affine group schemes of finite type.}

A grading on a Lie algebra $\cL$ by an abelian group $G$ is determined by a homomorphism
$\mathbf{D}(G)\rightarrow \AAut(\cL)$, where $\mathbf{D}(G)$ is the diagonalizable (and hence
reductive) group scheme whose representing Hopf algebra is the group algebra $\FF G$ 
(see, e.g., \cite[Chapter 1]{EKmon}). 

Vinberg  considered a large extension of this idea in his paper entitled \emph{Non-abelian gradings of Lie
algebras} \cite{Vinberg}, by substituting the diagonalizable groups above by arbitrary reductive groups.

\begin{definition}[{\cite[Definition 0.1]{Vinberg}}]\label{df:Sstructure}
Let $\mathbf{S}$ be a reductive algebraic group and let $\cL$ be a Lie algebra. An 
\emph{$\mathbf{S}$-structure} on $\cL$ is a homomorphism $\Phi:\mathbf{S}\rightarrow\AAut(\cL)$ from $\mathbf{S}$ into the algebraic group of automorphisms of $\cL$.
\end{definition}

Actually, this definition makes sense for nonassociative algebras, or more general algebraic systems, not just for Lie algebras.

Let $\mathbf{S}$ be a reductive algebraic group and let $\frs$ be its Lie algebra, the differential 
$\textup{d}\Phi$ of an $\mathbf{S}$-structure on the Lie algebra $\cL$ is a Lie algebra homomorphism 
$\textup{d}\Phi:\frs\rightarrow \Der(\cL)$. 

\begin{definition}\label{df:inner}
With the notations above, the $\mathbf{S}$-structure $\Phi:\mathbf{S}\rightarrow \AAut(\cL)$ on the Lie algebra $\cL$ is said to be \emph{inner} if there is a one-to-one Lie algebra homomorphism 
$\iota:\frs\hookrightarrow \cL$ such that the following diagram commutes: 
\begin{equation}\label{eq:iota}
\begin{tikzcd}
\frs\arrow[rd,"\textup{d}\Phi"']\arrow[r, hook,"\iota"]&\cL\arrow[d,"\ad"]\\
&\Der(\cL)
\end{tikzcd}
\end{equation}
\end{definition}

Note that if $\Phi:\mathbf{S}\rightarrow \AAut(\cL)$ is a non inner $\mathbf{S}$-structure on the 
Lie algebra $\cL$, then we can take the split extensions $\tilde\cL=\cL\oplus\frs$ as in 
\cite[p.~18]{Jacobson}, where $\frs$ acts on $\cL$ through $\textup{d}\Phi$, and this is endowed with a natural $\mathbf{S}$-structure. Hence it is not harmful to restrict to inner 
$\mathbf{S}$-structures.

\medskip

Definition \ref{df:Sstructure} is too general, so some restrictions on $\cL$ as a module for the
reductive group $\mathbf{S}$ must be imposed. Vinberg himself considered in \cite{Vinberg} two
different situations:
\begin{itemize}
\item A nontrivial $\SLs_2$-structure on a Lie algebra $\cL$ is called \emph{very short} if $\cL$ 
decomposes, as a module for $\SLs_2$, as a sum of copies of the adjoint module and of the trivial
one-dimensional module. Collecting isomorphic submodules, a very short $\SLs_2$-structure gives an 
isotypic decomposition of the form
\[
\cL=(\frsl_2\otimes\cJ)\oplus\cD
\]
and it turns out that the Lie bracket on $\cL$ induces a Jordan product on $\cJ$. The subalgebra $\cD$ acts by derivations on $\cJ$. All this goes back to \cite{Tits}. 

\item A nontrivial $\SLs_3$-structure on a Lie algebra $\cL$ is called \emph{short} if $\cL$ decomposes
as the direct sum of one copy of the adjoint representations, copies of its natural three-dimensional
module and of its dual, and copies of the trivial representation, so that the isotypic 
decomposition is:
\[
\cL=\frsl_3\oplus (V\otimes\cJ)\oplus (V^*\otimes \cJ')\oplus\cD.
\]
For simple $\cL$, $\cJ$ and $\cJ'$ may be identified, and inherit a structure of a cubic Jordan algebra
(see \cite[Equation (31)]{Vinberg}).  (A more general situation
was considered in \cite{BenkartElduque_sl3}.)
\end{itemize}

Stasenko \cite{Stasenko} has recently considered \emph{short $\SLs_2$-structures} on simple Lie algebras over the complex numbers (see Definition \ref{df:shortA1}). Although not with this 
terminology, the short $\SLs_2$-structures were considered in \cite{EO}. They are intimately related to the 
$J$-ternary algebras of Allison \cite{Allison}, a connection that will be reviewed in Section \ref{se:A1}.

{\rojo The reader should note that there is no general definition of short and very short structures. The definitions depend on the algebraic group $\mathbf{S}$ used (and on the inspiration of the different
authors to find a suitable name).}

\smallskip

Another important source of nice $\mathbf{S}$-structures is provided by the gradings by root systems,
initially considered by Berman and Moody \cite{BermanMoody} (see \cite{BenkartSmirnov} and the
references there in). 

In the simply-laced case, a Lie algebra graded by such a root system contains
a finite-dimensional split simple Lie algebra $\frs$ with such a root system, and it decomposes, 
as a module for $\frs$, as a direct sum of copies of the adjoint module and the trivial module,
so the corresponding isotypic decomposition has two components: $\cL=(\frs\otimes \cA)\oplus\cD$.
If $\mathbf{S}$ is the simply connected group with Lie algebra $\frs$, the action of $\frs$ integrates
to an $\mathbf{S}$-structure on $\cL$.

In the non simply-laced case, the isotypic decomposition also includes copies of the irreducible 
module for $\frs$ whose highest weight is the highest short root: 
$\cL=(\frs\otimes\cA)\oplus(W\otimes\cB)\oplus\cD$. Finally, the case of the nonreduced root systems
$BC_r$ gives isotypic decompositions with four components, with two exceptions
of five components (see \cite{AllisonBenkartGao02} and \cite{BenkartSmirnov}): $BC_1$-graded Lie 
algebras with grading subalgebra of type $D_1$ (which reduces to $5$-gradings), and $BC_2$-graded
Lie algebras with grading subalgebra of type $D_2=A_1\times A_1$.

\smallskip

Another nice class of $\mathbf{S}$-structures has been given in \cite{CunhaElduque}, where the
exceptional simple Lie algebras of type $E_r$, $r=7,8$, are shown to be endowed with a
$\SLs_2^r$-structure, such that the irreducible modules that appear are single copies of 
the adjoint modules for the factors of $\SLs_2^r$, tensor products of the two-dimensional natural
modules for some of these factors, and trivial modules. This allows us to coordinatize these
exceptional Lie algebras in terms of some algebras related to some well-known binary codes.

\medskip

The goal of this paper is to explore a new kind of $\mathbf{S}$-structures not covered by the
results mentioned above and where the reductive group is not simple: 
the short $(\SLs_2\times \SLs_2)$-structures  defined in Section 
\ref{se:A1A1}. These structures give isotypic decompositions with six components, because as a module for $\SLs_2\times\SLs_2$, the Lie algebra decomposes as a direct sum of 
copies of the adjoint modules for any of the two factors of $\SLs_2\times\SLs_2$, copies of the 
two-dimensional natural modules for each of these factors, copies of the tensor product of these two
natural modules, and copies of the trivial module. (See \eqref{eq:A1A1isotypic}.)

In particular, if we fix two of the factors $\SLs_2$ of the $\SLs_2^r$-structures in \cite{CunhaElduque},
a short $(\SLs_2\times\SLs_2)$-structure is obtained.

Trying to obtain directly the properties and multilinear operations among the six different components in 
\eqref{eq:A1A1isotypic} from the Lie bracket on the Lie algebra is quite cumbersome and not much
illuminating. However, as shown in Section \ref{se:A1A1}, the diagonal embedding $\SLs_2\rightarrow \SLs_2\times\SLs_2$ induces a short
$\SLs_2$-structure from any short $(\SLs_2\times\SLs_2)$-structure. This will be the clue to describe
the short $(\SLs_2\times\SLs_2)$-structures on a Lie algebra in terms of $J$-ternary algebras endowed with an idempotent (Theorem \ref{th:A1A1e}).

Section \ref{se:Examples} will be devoted to show examples of $J$-ternary algebras with idempotents 
and the corresponding short $(\SLs_2\times\SLs_2)$-structures on certain Lie algebras, and Section 
\ref{se:A1A1simple} will be devoted to using the results in \cite{Stasenko} to describe,
{\rojo in terms of the examples in  Section \ref{se:Examples},} all 
short $(\SLs_2\times\SLs_2)$-structures on the finite-dimensional simple Lie algebras over an
 algebraically closed field of characteristic $0$.

\smallskip

In concluding this introduction, let us mention that Lie algebras $\cL$ endowed with a homomorphism
$\Phi:\mathsf{G}\rightarrow \AAut(\cL)$, where $\mathsf{G}$ is a constant group scheme, have
been considered too in the literature and shown to be ``coordinatized'' by some nonassociative algebras
(see, e.g., \cite{EO,EOS3S4,EOS4}).

\bigskip

\section{Short \texorpdfstring{$\SLs_2$}{SL2}-structures and 
\texorpdfstring{$J$}{J}-ternary algebras}\label{se:A1}

Short $\SLs_2$-structures on simple Lie algebras over the complex numbers have been 
considered by Stasenko \cite[Definition 2]{Stasenko}. Hence we extend his definition over arbitrary
fields.

\begin{definition}\label{df:shortA1}
An $\SLs_2$-structure $\Phi:\SLs_2\rightarrow \AAut(\cL)$ on a Lie algebra $\cL$ is said to be \emph{short} if $\cL$ decomposes, as a module for $\SLs_2$ via $\Phi$, into a direct sum of
copies of the adjoint, natural, and trivial modules.
\end{definition}

Therefore, the isotypic decomposition of $\cL$ allows us to describe $\cL$ as follows:
\begin{equation}\label{eq:shortA1_isotypic}
\cL=\bigl(\frsl(V)\otimes \cJ\bigr)\oplus \bigl(V\otimes\cT\bigr)\oplus\cD,
\end{equation}
for vector spaces $\cJ$, $\cT$, and $\cD$, where $V$ is the natural two-dimensional representation of 
$\SLs_2\simeq\SLs(V)$. The action of $\SLs_2$ is given by the adjoint action of $\SLs_2$ on 
$\frsl(V)$, its natural action on $V$, and the trivial action on $\cJ$, $\cT$ and $\cD$. 
The subspace $\cD$, being the subspace of fixed elements by $\SLs_2$, is a subalgebra of $\cL$.

\smallskip

This section is devoted to reviewing the connection between short $\SLs_2$-structures on a Lie algebra and the $J$-ternary algebras of Allison \cite{Allison}, developed in \cite{EO}.

\smallskip

The following well known results will be needed. See, e.g., \cite[Lemma 2.1]{EO}.

\begin{lemma}\label{le:sl2}
Let $V$ be a two-dimensional vector space.
\begin{romanenumerate}
\item 
The space $\Hom_{\frsl(V)}\bigl(\frsl(V)\otimes\frsl(V),\frsl(V)\bigr)$ of $\frsl(V)$-invariant linear maps  
$\frsl(V)\otimes\frsl(V)\rightarrow\frsl(V)$ is spanned by the (skew-symmetric) Lie bracket:
\[
f\otimes g\mapsto [f,g]=fg-gf.
\]
\item 
The space $\Hom_{\frsl(V)}\bigl(\frsl(V)\otimes\frsl(V),\FF\bigr)$ is spanned by the trace map:
\[
f\otimes g\mapsto \trace(fg).
\]
\item 
The space $\Hom_{\frsl(V)}\bigl(\frsl(V)\otimes V,V\bigr)$ is spanned by the natural action:
\[
f\otimes v\mapsto f(v).
\]
\item 
The space $\Hom_{\frsl(V)}\bigl(V\otimes V,\FF\bigr)$ is one-dimensional. Its nonzero elements
are of the form:
\[
u\otimes v\mapsto (u\mid v)
\]
for a nonzero skew-symmetric bilinear form $(.\mid .)$ on $V$. 
\item
The space $\Hom_{\frsl(V)}\bigl(V\otimes V,\frsl(V)\bigr)$ is one-dimensional. Once a nonzero
skew-symmetric bilinear form $(.\mid .)$ is fixed on $V$, this subspace is spanned by the following symmetric map:
\[
u\otimes v\mapsto \gamma_{u,v}\bigl(: w\mapsto (u\mid w)v+(v\mid w)u\bigr).
\]
\item
The spaces $\Hom_{\frsl(V)}\bigl(\frsl(V)\otimes\frsl(V),V\bigr)$, 
$\Hom_{\frsl(V)}\bigl(\frsl(V)\otimes V,\frsl(V)\bigr)$,  \\
$\Hom_{\frsl(V)}\bigl(\frsl(V)\otimes V,\FF\bigr)$, {\rojo and
$\Hom_{\frsl(V)}\bigl(V\otimes V,V\bigr)$} are all trivial.
\end{romanenumerate}

Moreover, $\Hom_{\frsl(V)}$ may be replaced by $\Hom_{\SLs(V)}$ all over.
\end{lemma}

\begin{remark}
Lemma \ref{le:sl2} is even valid in characteristic $3$. Actually, if the characteristic of our ground field
$\FF$ is $3$, then \cite[Theorem 1.11]{BO} gives the following isomorphisms of $\frsl(V)$-modules:
\[
V\otimes V\simeq \frsl(V)\oplus \FF,\quad \frsl(V)\otimes V\simeq Q(1),\quad
\frsl(V)\otimes \frsl(V)\simeq \frsl(V)\oplus Q(0),
\]
where $Q(1)$ is an indecomposable module with a unique maximal submodule $W$ such that
$Q(1)/W$ is isomorphic to $V$, and $Q(0)$ is an indecomposable module with a unique maximal
submodule $W$ such that $Q(0)/W$ is isomorphic to the trivial module $\FF$. It follows at once
that all the spaces in items \textbf{(i)--(v)} are one-dimensional, and those in item 
\textbf{(vi)} are trivial.
\end{remark}

{\rojo \emph{Let us return to our assumption that the characteristic of the ground field is $\neq 2,3$. This assumption will be kept throughout the paper with no further mention.}}

\smallskip

From now on, we will fix a nonzero skew-symmetric bilinear form $(.\mid .)$ on our two-dimensional vector space $V$.

Let $\cL$ be a Lie algebra with an inner $\SLs_2$-structure and isotypic decomposition as in 
\eqref{eq:shortA1_isotypic}. The $\SLs_2$-structure being inner forces $\cJ$ to contain a distinguished element $1$, such that $\frsl(V)\otimes 1$ is the image of $\iota$ in 
\eqref{eq:iota}.

{\rojo The $\SLs_2$-invariance or, equivalently, the $\frsl(V)$-invariance,} of the Lie bracket in our Lie algebra $\cL$ gives, for any $f,g\in\frsl(V)$, $u,v\in V$ and $D\in \cD$, the following conditions:
\begin{equation}\label{eq:shortA1bracket}
\begin{split}
[f\otimes a,g\otimes b]&=[f,g]\otimes a\cdot b+2\trace(fg)D_{a,b},\\
[f\otimes a,u\otimes x]&=f(u)\otimes a\bullet x,\\
[u\otimes x,v\otimes y]&=\gamma_{u,v}\otimes \langle x\mid y\rangle 
                                              +\bigl(u\mid v\bigr)d_{x,y},\\
[D,f\otimes a]&=f\otimes D(a),\\
[D,u\otimes x]&=u\otimes D(x),
\end{split}
\end{equation}
for suitable bilinear maps
\begin{equation}\label{eq:shortA1maps}
\begin{split}
\cJ\times \cJ\rightarrow \cJ,&\quad (a,b)\mapsto a\cdot b,\quad\text{(symmetric),}\\
\cJ\times \cJ\rightarrow \cD,&\quad (a,b)\mapsto D_{a,b},\quad\text{(skew-symmetric),}\\
\cJ\times \cT\rightarrow \cT,&\quad (a,x)\mapsto a\bullet x,\\
\cT\times \cT\rightarrow \cJ,&\quad (x,y)\mapsto \langle x\mid y\rangle\quad
                     \text{(skew-symmetric),}\\
\cT\times \cT\rightarrow \cD,&\quad (x,y)\mapsto d_{x,y}\quad\text{(symmetric),}\\
\cD\times \cJ\rightarrow \cJ,&\quad (D,a)\mapsto D(a),\\
\cD\times \cT\rightarrow \cT,&\quad (D,x)\mapsto D(x),
\end{split}
\end{equation}
such that 
{\rojo
\begin{equation}\label{eq:D1a}
1\cdot a=a,\quad D_{1,a}=0,\quad\text{and}\quad 1\bullet x=x,
\end{equation}
for any $a,b\in \cJ$, $x,y\in \cT$, 
 and $D\in\cD$.
}

The Jacobi identity on $\cL$ also shows that all these maps are invariant under the action of the Lie subalgebra $\cD$. The next result summarizes the properties of these maps:

\begin{theorem}[{\cite[Theorem 2.2]{EO}}]\label{th:ShortA1Properties}
A Lie algebra $\cL$ is endowed with an inner short $\SLs_2$-structure if and only if there is a 
two-dimensional vector space $V$ such that $\cL$ is, up to isomorphism, the Lie algebra in \eqref{eq:shortA1_isotypic}, with Lie bracket given in \eqref{eq:shortA1bracket}, for suitable 
bilinear maps given in \eqref{eq:shortA1maps}, satisfying the following conditions:
\begin{itemize}
\item
$\cJ$ is a unital Jordan algebra with the multiplication $a\cdot b$.
\item
$\cT$ is a special unital Jordan module for $\cJ$ with the action $a\bullet x$. That is, the map
$\cJ\rightarrow \End(\cT)^{(+)}$, given by $a\mapsto (x\mapsto a\bullet x)$, is a homomorphism of unital Jordan algebras. (For an associative algebra $\cA$, $\cA^{(+)}$ denotes
the special Jordan algebra with multiplication $a\cdot b=\frac{1}{2}(ab+ba)$.)

In other words, the following equation holds for $a,b\in\cJ$ and $x\in \cT$:
\begin{equation}\label{eq:bullet}
(a\cdot b)\bullet x=\frac{1}{2}\bigl(a\bullet(b\bullet x)+b\bullet(a\bullet x)\bigr).
\end{equation}
\item
For any $a,b,c\in\cJ$ and $x,y,z\in\cT$, the following identities hold:
\begin{gather}
D_{a,b}(c)=a\cdot (b\cdot c)-b\cdot (a\cdot c),\label{eq:Dab}\\
D_{a\cdot b,c}+D_{b\cdot c,a}+D_{c\cdot a,b}=0,\label{eq:Dabc}\\
4D_{a,b}(x)=a\bullet(b\bullet x)-b\bullet(a\bullet x),\\
4D_{a,\langle x\mid y\rangle}=-d_{a\bullet x,y}+d_{x,a\bullet y},\label{eq:Dd}\\
2a\cdot \langle x\mid y\rangle=\langle a\bullet x\mid y\rangle +\langle x \mid a\bullet y\rangle,
\label{eq:axy}\\
d_{x,y}(a)=\langle a\bullet x\mid y\rangle -\langle x\mid a\bullet y\rangle,\\
d_{x,y}(z)-d_{z,y}(x)=\langle x\mid y\rangle \bullet z-\langle z\mid y\rangle\bullet x
  +{\rojo 2\langle x\mid z\rangle\bullet y}.
\end{gather}
\item
For any $D\in\cD$, the linear endomorphism of $\cJ\oplus \cT$, given by $a+x\mapsto D(a)+D(x)$ for
$a\in\cJ$ and $x\in\cT$, is an even derivation of the $\ZZ/2$-graded algebra with even part $\cJ$, odd part $\cT$, and multiplication given by the formula:
\begin{equation}\label{eq:J+T}
(a+x)\diamond(b+y)=\bigl(a\cdot b+\langle x\mid y\rangle\bigr)+\bigl(a\bullet y+b\bullet x\bigr),
\end{equation}
for $a,b\in\cJ$ and $x,y\in\cT$.\qed
\end{itemize}
\end{theorem}

As remarked in \cite{EO}, all this is strongly related to the $J$-ternary algebras considered by 
Allison \cite{Allison}.

\begin{definition}[{\cite[(3.12)]{AllisonBenkartGao02}}]\label{df:Jternary} Let $\cJ$ be a unital Jordan algebra with multiplication $a\cdot b$ for $a,b\in\cJ$. Let $\cT$ be a unital special Jordan module for $\cJ$ with action $a\bullet x$ for $a\in \cJ$ and $x\in \cT$. Assume $\langle.\mid.\rangle:\cT\times \cT\rightarrow \cJ$ is a skew-symmetric bilinear map and $\ptriple{.,.,.}:\cT\times \cT\times \cT\rightarrow \cT$ is a trilinear product on $\cT$. Then the pair $(\cJ,\cT)$ is called a \emph{$J$-ternary algebra} if the following axioms hold for any $a\in \cJ$ and $x,y,z,w,v\in \cT$:
\begin{description}
\addtolength{\itemindent}{-22pt}
\item[(JT1)] $a\cdot\langle x\mid y\rangle =\dfrac{1}{2}\bigl(\langle a\bullet x\mid y\rangle 
+\langle x\mid a\bullet y\rangle\bigr)$,

\item[(JT2)] $a\bullet\ptriple{ x,y,z}=\ptriple{ a\bullet x,y,z}-\ptriple{ x,a\bullet y,z} +\ptriple{x,y,a\bullet z}$,

\item[(JT3)] $\ptriple{x,y,z}=\ptriple{z,y,x}-\langle x\mid z\rangle\bullet y$,

\item[(JT4)] $\ptriple{x,y,z}=\ptriple{y,x,z}+\langle x\mid y\rangle\bullet z$,

\item[(JT5)] $\langle\ptriple{x,y,z}\mid w\rangle+\langle z\mid\ptriple{x,y,w}\rangle =\langle x\mid\langle z\mid w\rangle\bullet y\rangle$,

\item[(JT6)] $\ptriple{x,y,\ptriple{z,w,v}}=\ptriple{\ptriple{x,y,z},w,v}+
    \ptriple{z,\ptriple{y,x,w},v}+\ptriple{z,w,\ptriple{x,y,{\rojo v}}}$.
\end{description}
\end{definition}

Theorem 2.4 in \cite{EO} becomes now the following result:

\begin{theorem}\label{th:SL2Jternary}
Let $\cL$ be a Lie algebra endowed with {\rojo a short} inner $\SLs_2$-structure with isotypic decomposition in 
\eqref{eq:shortA1_isotypic}. Then the pair $(\cJ,\cT)$ is a $J$-ternary algebra with the triple product
on $\cT$ given by the next formula
\begin{equation}\label{eq:Jternary}
\ptriple{x,y,z}=\frac{1}{2}\Bigl(-d_{x,y}(z)+\langle x\mid y\rangle\bullet z\Bigr).
\end{equation}
for all $x,y,z\in\cT$, {\rojo where $\langle\,\mid\,\rangle$, $\bullet$, and $d_{x,y}$ are defined
as in \eqref{eq:shortA1bracket} and \eqref{eq:shortA1maps}.}

Conversely, if $(\cJ,\cT)$ is a $J$-ternary algebra {\rojo with bilinear maps $a\cdot b$, $a\bullet x$, $\langle x\mid y\rangle$, and trilinear map $\ptriple{x,y,z}$, as in Definition \ref{df:Jternary}}, 
then the vector space
\begin{equation}\label{eq:LJT}
\cL(\cJ,\cT)\bydef (\frsl(V)\otimes\cJ)\oplus(V\otimes\cT)\oplus\cD,
\end{equation}
is a Lie algebra with an inner $\SLs_2$-structure with the bracket defined as in \eqref{eq:shortA1bracket}, and
where $\cD$ is the subalgebra of the Lie algebra of even derivations of the $\ZZ/2$-graded algebra
$\cJ\oplus\cT$ defined in \eqref{eq:J+T} spanned  by the maps $D_{a,b}$ for $a,b\in \cJ$, and $d_{x,y}$ for $x,y\in\cT$, defined as follows:
\begin{equation}\label{eq:defDd}
\begin{split}
D_{a,b}(c)&=a\cdot(b\cdot c)-b\cdot(a\cdot c),\\
D_{a,b}(x)&=\frac{1}{4}\bigl(a\bullet(b\bullet x)-b\bullet (a\bullet x)\bigr),\\
d_{x,y}(a)&=\langle a\bullet x\mid y\rangle -\langle x\mid a\bullet y\rangle,\\
d_{x,y}(z)&=\langle x\mid y\rangle\bullet z-2\ptriple{x,y,z},
\end{split}
\end{equation}
for $a,b,c\in\cJ$ and $x,y,z\in\cT$.
\end{theorem}

Under the conditions of Theorem \ref{th:SL2Jternary}, the $J$-ternary algebra $(\cJ,\cT)$ is said to \emph{coordinatize} the $\SLs_2$-structure on $\cL$.

\smallskip

{\rojo A word of caution is needed here. If $\cL$ is a Lie algebra endowed with an inner 
$\SLs_2$-structure as
in  Theorem \ref{th:SL2Jternary}, and $(\cJ,\cT)$ is the associated $J$-ternary algebra that 
coordinatizes it, the
Lie algebra $\cL(\cJ,\cT)$ is not necessarily isomorphic to our original $\cL$. It may even fail to be isomorphic to a subalgebra of $\cL$. Actually, $\cL(\cJ,\cT)$  is  \emph{centrally isogenous} 
to the subalgebra of the original $\cL$
generated by the isotypic components $\frsl(V)\otimes\cJ$ and $V\otimes\cT$. That is, their universal 
central extensions coincide (see \cite{BermanMoody,AllisonBenkartGao02} and the references
therein). We will not go into
details.}

\bigskip

\section{Short \texorpdfstring{$(\SLs_2\times\SLs_2)$}{SL2xSL2}-structures}\label{se:A1A1}

We start by introducing a new class of $\mathbf{S}$-structures: the short 
$(\SLs_2\times\SLs_2)$-structures.

\begin{definition}\label{df:shortA1A1}
An $(\SLs_2\times\SLs_2)$-structure $\Phi:\SLs_2\times\SLs_2\rightarrow \AAut(\cL)$ on a Lie algebra 
$\cL$ is said to be \emph{short} if $\cL$ decomposes, as a module for $\SLs_2\times\SLs_2$ via 
$\Phi$, into a direct sum of copies of the following modules:
\begin{itemize}
\item the adjoint module for each of the two copies of $\SLs_2$, 
\item the natural two-dimensional modules $V_1$ and $V_0$ for each of the two copies of $\SLs_2$
(the weird numbering $1,0$ is justified by the Peirce decomposition in \eqref{eq:Peirce}),
\item the tensor product $V_1\otimes V_0$, and
\item the trivial one-dimensional module.
\end{itemize}
\end{definition}

{\rojo
\begin{remark}
Due to Lemma \ref{le:sl2}, where $\SLs_2$-invariance may be substituted by $\frsl_2$-invariance,
the reader who feels uncomfortable with affine group schemes may consider an alternate 
definition of \emph{short $\frsl_2\oplus\frsl_2$-structure} on a Lie algebra $\cL$, as a homomorphism
of Lie algebras $\Psi:\frsl_2\oplus\frsl_2\rightarrow \Der(\cL)$ such that, as a module for
$\frsl_2\oplus\frsl_2$ via $\Psi$, $\cL$ decomposes into a direct sum of copies of the adjoint module
for each of the two copies of $\frsl_2$, of the natural two-dimensional modules $V_1$ and $V_0$
 for each of the two
copies of $\frsl_2$, of the tensor product $V_1\otimes V_0$, and of the trivial one-dimensional
module.
\end{remark}}

\medskip

{\rojo Assume that  $\Phi:\SLs_2\times\SLs_2\rightarrow \cL$ is an inner short $(\SLs_2\times\SLs_2)$-structure 
on the Lie algebra $\cL$, then 
its isotypic decomposition} allows us to describe $\cL$ as follows:
\begin{equation}\label{eq:A1A1isotypic}
\cL=\bigl(\frsl(V_1)\otimes \cJ_1\bigr)\oplus
      \bigl(\frsl(V_0)\otimes \cJ_0\bigr)\oplus
      \bigl((V_1\otimes V_0)\otimes \cJ_{\frac12}\bigr)\oplus
      \bigl(V_1\otimes \cT_1\bigr)\oplus
      \bigl(V_0\otimes \cT_0\bigr)\oplus \cS.
\end{equation}

Note that the invariance of the bracket under the action of $\SLs_2\times\SLs_2\simeq
\SLs(V_1)\times\SLs(V_0)$ forces that $\bigl(\frsl(V_i)\otimes \cJ_i\bigr)\oplus
 \bigl(V_i\otimes \cT_i\bigr)\oplus \cS$ is a subalgebra with a short inner $\SLs_2$-structure, for $i=0,1$,
and hence that the bracket in $\cL$ induces a structure of $J$-ternary algebra on $(\cJ_i,\cT_i)$ for 
$i=1,0$.

In order to get more information, {\rojo it is not a good idea to expand blindly the Lie bracket
on $\cL$. Instead, consider} the diagonal subgroup of $\SLs_2\times\SLs_2$.
The composition
\begin{equation}\label{eq:Delta}
\SLs_2\xrightarrow{\Delta}\SLs_2\times\SLs_2\xrightarrow{\Phi}\AAut(\cL),
\end{equation}
where $\Delta$ is the diagonal embedding $g\mapsto (g,g)$, 
gives an inner $\SLs_2$-structure on $\cL$. This structure is short because, as $\SLs_2$-modules
via $\Phi\circ\Delta$, $\frsl(V_i)$ is the adjoint module and $V_i$ is the two-dimensional natural module, for $i=1,0$, and $V_1\otimes V_0$ decomposes as the direct sum of an adjoint module and a trivial module. {\rojo (See items (iv) and (v) of Lemma \ref{le:sl2}.)}

The next theorem is our main result. It describes the short inner $(\SLs_2\times\SLs_2)$-structures
on a Lie algebra {\rojo in terms of} $J$-ternary algebras $(\cJ,\cT)$ where the unital Jordan algebra $\cJ$ contains a nontrivial distinguished idempotent.

\begin{theorem}\label{th:A1A1e}
Let $\cL$ be a Lie algebra endowed with {\rojo a short} inner $(\SLs_2\times\SLs_2)$-structure 
$\Phi:\SLs_2\times\SLs_2\rightarrow \AAut(\cL)$. Let $(\cJ,\cT)$ be the $J$-ternary algebra that
coordinatizes the short $\SLs_2$-structure $\Psi=\Phi\circ\Delta$ in \eqref{eq:Delta}, with isotypic decomposition given by \eqref{eq:shortA1_isotypic}:
\begin{equation}\label{eq:shortA1_isotypic_bis}
\cL=\bigl(\frsl(V)\otimes \cJ\bigr)\oplus \bigl(V\otimes\cT\bigr)\oplus\cD,
\end{equation}
where $V$ is the two-dimensional natural module for $\SLs_2$.

Then the
unital Jordan algebra $\cJ$ contains an idempotent $e=e^{\cdot 2}\neq 0,1$ such that the image
of the Lie algebra homomorphism $\iota$ in  \eqref{eq:iota} is 
\[
\bigl(\frsl(V)\otimes e\bigr)\oplus\bigl(\frsl(V)\otimes (1-e)\bigr).
\]

Conversely, if $\cL$ is a Lie algebra endowed with an inner $\SLs_2$-structure 
$\Psi:\SLs_2\rightarrow\AAut(\cL)$ with isotypic decomposition in \eqref{eq:shortA1_isotypic_bis}, 
and such that the unital Jordan algebra $\cJ$ contains a nontrivial idempotent $e$, 
then $\cL$ is endowed with a short inner $(\SLs_2\times\SLs_2)$-structure whose
isotypic decomposition is the following:
\begin{multline}\label{eq:A1A1_e_isotypic}
\cL=\bigl(\frsl(V)\otimes\cJ_1\bigr)\oplus\bigl(\frsl(V)\otimes\cJ_0\bigr)\oplus
    \Bigl(\bigl(\frsl(V)\otimes\cJ_{\frac12}\bigr)\oplus D_{e,\cJ_{\frac12}}\Bigr)\\
    \oplus
    \bigl(V\otimes\cT_1\bigr)\oplus\bigl(V\otimes\cT_0\bigr)\oplus\cS,
\end{multline}
where $\cJ=\cJ_1\oplus\cJ_{\frac12}\oplus\cJ_0$ is the Peirce decomposition of $\cJ$ relative to the idempotent $e$:
\begin{equation}\label{eq:Peirce}
\cJ_1=\{a\in\cJ\mid e\cdot a=a\},\quad 
\cJ_{\frac12}=\{a\in\cJ\mid e\cdot a=\frac12 a\},\quad
\cJ_0=\{a\in\cJ\mid e\cdot a=0\},
\end{equation}
$\cT=\cT_1\oplus\cT_0$ is the induced decomposition on $\cT$:
\begin{equation}\label{eq:TT0T1}
\cT_1=\{x\in\cT\mid e\bullet x=x\},\quad \cT_0=\{x\in\cT\mid e\bullet x=0\},
\end{equation}
and where the following assertions hold:
\begin{itemize}
\item 
for any $a\in\cJ_1$, $\frsl(V)\otimes a$ is a copy of the adjoint module for the first copy of 
$\SLs_2$,
\item 
 for any $a\in\cJ_0$, $\frsl(V)\otimes a$ is a copy of the adjoint module for the second copy of 
$\SLs_2$,
\item
 for any $a\in\cJ_{\frac12}$, $\bigl(\frsl(V)\otimes a\bigr)\oplus\FF D_{e,a}$ is a copy of the tensor product of the natural module for the first copy of 
$\SLs_2$ and the natural module for the second copy of $\SLs_2$,
\item
 for any $x\in\cT_1$, $V\otimes x$ is a copy of the natural module for the first copy of 
$\SLs_2$,
\item
 for any $x\in\cT_0$, $V\otimes x$ is a copy of the natural module for the second copy of 
$\SLs_2$,
\item
$\cS$ is the subspace of fixed elements by $\SLs_2\times\SLs_2$. Moreover, the subspace
of fixed elements by $\SLs_2$ under $\Psi=\Phi\circ\Delta$ is the direct sum 
$\cD=\cS\oplus D_{e,\cJ_{\frac12}}$.
\end{itemize}
\end{theorem}
\begin{proof}
For the first part, as the $(\SLs_2\times\SLs_2)$-structure $\Phi$ is inner, $\cL$ contains a subalgebra isomorphic to 
$\frsl_2\oplus\frsl_2$, which in the decomposition \eqref{eq:shortA1_isotypic_bis} appears as
\[
\bigl(\frsl(V)\otimes e_1\bigr)\oplus\bigl(\frsl(V)\otimes e_2\bigr),
\]
for orthogonal idempotents $e_1,e_2\in \cJ$; that is, $e_1^{\cdot 2}=e_1$, 
$e_2^{\cdot 2}=e_2$, $e_1\cdot e_2=0$ and $1=e_1+e_2$. Now it is enough to take $e=e_1$.

Conversely, let $\cL$ be a Lie algebra endowed with {\rojo a short} inner $\SLs_2$-structure 
$\Psi:\SLs_2\rightarrow\AAut(\cL)$ with isotypic decomposition as in \eqref{eq:shortA1_isotypic_bis}.
Let $e\in\cJ$ be a nontrivial idempotent (i.e., $e=e^{\cdot 2}\neq 0,1$), and let $\cJ=\cJ_1\oplus\cJ_{\frac12}\oplus\cJ_0$ be the corresponding Peirce decomposition as in \eqref{eq:Peirce}.
Equation \eqref{eq:bullet} gives $e\bullet x=e\bullet(e\bullet x)$ for any $x\in\cT$, so that 
$\cT$ decomposes as in \eqref{eq:TT0T1}.

\smallskip

We need a short digression before continuing with the proof. The vector space $\frgl(V)$ of linear endomorphisms of our two-dimensional vector space $V$ is a module for $\frsl(V)\oplus\frsl(V)$ by means of the action
\begin{equation}\label{eq:sl2sl2action}
(f_1,f_2)\cdot g\bydef f_1 g-g f_2
\end{equation}
for $f_1,f_2\in\frsl(V)$ and $g\in\frgl(V)$. Moreover, once we fix, as we have done before, a nonzero 
skew-symmetric bilinear form $(.\mid .)$ on $V$, $\frgl(V)$ is isomorphic to $V\otimes V$ by means of the linear map determined as follows:
\[
u\otimes v\mapsto u(v\mid .)\bigl(: w\mapsto (v\mid w)u\bigr),
\]
for $u,v,w\in V$. This linear isomorphism is an isomorphism of $\frsl(V)\oplus\frsl(V)$-modules,
where the first (resp. second) copy of $\frsl(V)$ acts on the first (resp. second) copy of $V$ in 
$V\otimes V$.

Also note that for any $f\in \frsl(V)$, the Cayley-Hamilton equation gives $f^2=-\det(f)\id=\frac12\trace(f^2)\id$, and hence, by linearization we get
\begin{equation}\label{eq:fg+gf}
fg+gf=\trace(fg)\id
\end{equation}
for any $f,g\in\frsl(V)$. On the other hand, $[f,g]=fg-gf$ which, together with \eqref{eq:fg+gf} gives
the formulas:
\begin{equation}\label{eq:fggf}
fg=\frac12\bigl([f,g]+\trace(fg)\id\bigr), \quad gf=\frac12\bigl(-[f,g]+\trace(fg)\id\bigr)
\end{equation}
for $f,g\in\frsl(V)$, and hence we have
\begin{equation}\label{eq:f1g-gf2}
(f_1,f_2)\cdot g=f_1g-gf_2=\frac12\bigl([{\rojo f_1+f_2},g]+\trace(({\rojo f_1-f_2})g)\id\bigr)
\end{equation}
for any $f_1,f_2,g\in\frsl(V)$.

(Note that $\frsl(V)\oplus\frsl(V)$ can be substituted by $\SLs(V)\times\SLs(V)$ above, with
\eqref{eq:sl2sl2action} changed to $(f_1,f_2)\cdot g=f_1gf_2^{-1}$, for rational points 
$f_1,f_2\in\SLs(V)$ and $g\in\frgl(V)$.)

\smallskip

Now, for $\cL$ as above, endowed with a short inner $\SLs_2$-structure $\Psi$, \eqref{eq:shortA1bracket} gives, for any $f_1,f_2,g\in\frsl(V)$ and $a\in\cJ_{\frac12}$, the following:
\[
\begin{split}
[f_1\otimes e+f_2\otimes(1-e),g\otimes a]&={\rojo \frac12[f_1+f_2,g]\otimes a}+2\trace((f_1-f_2)g)D_{e,a},\\
[f_1\otimes e+f_2\otimes(1-e),D_{e,a}]&=\frac14 (f_1-f_2)\otimes a,
\end{split}
\]
{\rojo where we have used $D_{1,a}=0$ and $D_{e,a}(e)=-\frac{1}{4}a$, which follow from \eqref{eq:D1a} and \eqref{eq:Dab}.}

Comparing with \eqref{eq:f1g-gf2}, it turns out that the linear map
\[
\begin{split}
\bigl(\frsl(V)\otimes a\bigr)\oplus \FF D_{e,a}&\longrightarrow \frgl(V)\\
f\otimes a\qquad &\mapsto\quad f,\\
D_{e,a}\qquad &\mapsto\quad \frac14 \id,
\end{split}
\]
is an isomorphism of $\frsl(V)\oplus \frsl(V)\simeq \bigl(\frsl(V)\otimes e\bigr)\oplus
\bigl(\frsl(V)\otimes (1-e)\bigr)$-modules.

On the other hand, for any $a\in\cJ_{\frac12}$, Equation \eqref{eq:Dab} gives 
$D_{e,a}(e)=\left(\frac14 -\frac12\right) a=-\frac14 a$, so the linear map 
\[
\cJ_{\frac12}\rightarrow \cD,\quad a\mapsto D_{e,a}
\]
is one-to-one, and $\cD$ decomposes as $\cD=D_{e,\cJ_{\frac12}}\oplus \cS$, with 
\[
\cS=\{D\in\cD\mid D(e)=0\}.
\]
It then follows that $\cS$ is the centralizer of the subalgebra
$\bigl(\frsl(V)\otimes e\bigr)\oplus
\bigl(\frsl(V)\otimes (1-e)\bigr)\simeq \frsl(V)\oplus \frsl(V)$.

Therefore, the decomposition in \eqref{eq:A1A1_e_isotypic} gives a decomposition of $\cL$, as a module for $\frsl_2\oplus \frsl_2\simeq \bigl(\frsl(V)\otimes e\bigr)\oplus
\bigl(\frsl(V)\otimes (1-e)\bigr)$,
into a direct sum of copies of the adjoint module for the first copy of
$\frsl_2$, copies of the adjoint module for the second copy of $\frsl_2$, copies of the tensor product 
of the natural modules for the two copies of $\frsl_2$, copies of the natural module for the first copy of $\frsl_2$, copies for the natural module for the second copy of $\frsl_2$, and copies of the trivial one-dimensional module for $\frsl_2\oplus\frsl_2$. This shows that $\cL$ is endowed with a short inner $(\SLs_2\times\SLs_2)$-structure.
\end{proof}

\begin{remark}\label{re:consequences}
Let $\cL$ be a Lie algebra endowed with {\rojo a short} inner $\SLs_2$-structure 
$\Psi:\SLs_2\rightarrow\AAut(\cL)$ with isotypic decomposition as in \eqref{eq:shortA1_isotypic_bis}, 
and such that the unital Jordan algebra $\cJ$ contains a nontrivial idempotent $e$. Then 
\eqref{eq:Dabc} with $c=e$, $a\in \cJ_1$, and $b\in\cJ_0$ gives 
\[
D_{\cJ_1,\cJ_0}=0.
\]
and now, with $b=c=e$ and $a\in \cJ_1$, gives $D_{e,\cJ_1}=0$. {\rojo (This should be familiar to 
experts in Jordan algebras, but note that here $D_{a,b}$ is an element of the subalgebra $\cD$, which
acts by derivations on $\cJ$, but it is not, in general, the Lie algebra of derivations of $\cJ$.)}

Also, \eqref{eq:axy} with $a=e$ gives:
\[
\langle \cT_i\mid\cT_i\rangle\subseteq \cJ_i\ (i=1,0),\qquad \langle\cT_1\mid\cT_0\rangle \subseteq\cJ_{\frac12},
\]
and \eqref{eq:Dd} with $a=e$ shows $d_{x.y}=4D_{e,\langle x\mid y\rangle}$ for $x\in\cT_0$ and $y\in\cT_1$, which implies 
\[
d_{\cT_0,\cT_1}\subseteq D_{e,\cJ_{\frac12}}.
\]
All these conditions follow too from the fact that the bracket in $\cL$ is invariant under the action of
$\bigl(\frsl(V)\otimes e\bigr)\oplus\bigl(\frsl(V)\otimes(1-e)\bigr)$.
\end{remark}

\begin{remark}\label{re:BC2C2graded}
The Lie algebras graded by the nonreduced root system $BC_2$ and with grading subalgebra of type
$C_2$ are also coordinatized by $J$-ternary algebras with a proper idempotent plus some extra 
restrictions (\cite[Theorem 6.66]{AllisonBenkartGao02}). The connection with Theorem \ref{th:A1A1e} 
is as follows. Take two two-dimensional vector spaces $V_1$, $V_2$, endowed with nonzero alternating 
forms. Identify the simple Lie algebra of type $C_2$ with $\frs=\frsp(V_1\perp V_2)$. Then $\frsl(V_1)\oplus\frsl(V_2)\simeq \frsl_2\oplus\frsl_2$ is naturally a subalgebra of $\frs$, and $\frs$ decomposes,
as a module for $\frsl(V_1)\oplus\frsl(V_2)$, as $\frsl(V_1)\oplus\frsl(V_2)\oplus (V_1\otimes V_2)$.
The $\frs$-modules that appear in a $BC_2$-graded Lie algebra with grading subalgebra $\frs$
decompose, as modules for $\frsl(V_1)\oplus\frsl(V_2)$, into a direct sum of copies of the adjoint 
modules $\frsl(V_1)$ and $\frsl(V_2)$, natural modules $V_1$ and $V_2$, $V_1\otimes V_2$, and
the trivial module. Hence any $BC_2$-graded Lie algebra with a grading subalgebra of type $C_2$ is
naturally
endowed with a short $(\SLs_2\times\SLs_2)$-structure.
\end{remark}

\bigskip

\section{Examples}\label{se:Examples}

In this section, several examples of $J$-ternary algebras will be {\rojo reviewed}. Any nontrivial idempotent in the corresponding Jordan algebras provides examples of short 
$(\SLs_2\times\SLs_2)$-structures on the associated Lie algebra. {\rojo The first class of examples of $J$-ternary algebras are defined in terms of  unital associative algebras with involution 
and  left modules for them, while the second class of examples is related to structurable algebras.}

The classification of the reduced finite-dimensional $J$-ternary algebras was obtained by W.~Hein 
\cite{Hein1,Hein2}.

\smallskip

\subsection{Prototypical example}\label{ss:proto}

Let $(\cA,*)$ be {\rojo a unital} associative algebra with involution, and let $\cT$ be a left $\cA$-module endowed with a skew-hermitian form $\hup:\cT\times\cT\rightarrow \cA$. That is, $\hup$ is $\FF$-bilinear, and   
$\hup(ax,y)=a\hup(x,y)$  and $\hup(x,y)={\rojo -\hup(y,x)^*}$, for any $a\in\cA$ and $x,y\in\cT$.

As above, let $V$ be a two-dimensional vector space endowed with a nonzero skew-symmetric
bilinear form $(.\mid .)$, and consider the left $\cA$-module 
{\rojo
\[
\cW=(V\otimes \cA)\oplus\cT,
\]}%
where $V\otimes \cA$ is a left $\cA$-module in the natural way: $a(u\otimes b)\bydef u\otimes (ab)$, for any $a,b\in \cA$ and $u\in V$. Extend $\hup$ to a skew-hermitian form, also denoted by $\hup$, on $\cW$ by imposing that 
$V\otimes\cA$ and $\cT$ are orthogonal relative to $\hup$, and defining its restriction to $V\otimes \cA$ as follows:
\begin{equation}\label{eq:huavb}
\hup(u\otimes a,v\otimes b)\bydef 2(u\mid v)ab^*,
\end{equation}
for any $u,v\in V$ and $a,b\in\cA$.

Denote by $\tau$ the involution  on $\End_\cA(\cW)$ induced by $\hup$:
\[
\hup\bigl(f(x),y\bigr)=\hup\bigl(x,f^\tau(y)\bigr)
\]
for any $f\in\End_\cA(\cW)$ and any $x,y\in\cW$, and denote too by $\tau$ the restriction of $\tau$ to $\End_\cA(\cT)$.

For any $x,y\in\cW$, the $\cA$-linear map defined as follows:
\begin{equation}\label{eq:phixy}
\varphi_{x,y}=\hup(.,x)y+\hup(.,y)x: z\mapsto \hup(z,x)y+\hup(z,y)x,
\end{equation}
is skew-symmetric relative to $\tau$. That is, $\varphi_{x,y}^\tau=-\varphi_{x,y}$, or 
$\varphi_{x,y}$ belongs to the Lie algebra $\Skew\bigl(\End_\cA(\cW),\tau\bigr)$. Also, for any $a\in\cA$ and $x,y\in\cW$, we have
\[
\varphi_{ax,y}=\varphi_{x,a^*y}
\]
because $\hup$ is skew-hermitian. And for any $f\in\Skew\bigl(\End_\cA(\cW),\tau\bigr)$ and
any $x,y\in\cW$, the following equation follows at once:
\begin{equation}\label{eq:fvphixy}
[f,\varphi_{x,y}]=\varphi_{f(x),y}+\varphi_{x,f(y)}.
\end{equation}

{\rojo For any endomorphism $\varphi$ of the left regular module $\cA$, there is an element $a\in \cA$
such that $\varphi(1)=a^*$. Hence, for any $b\in\cA$, we get $\varphi(b)=\varphi(b1)=b\varphi(1)=ba^*$.
As a consequence, the 
associative algebra $\End_\cA(\cA)$ is isomorphic to $\cA$ by means of the assignment:
\begin{equation}\label{eq:EndAAA}
a\mapsto \bigl(R_{a^*}:b\mapsto ba^*\bigr),
\end{equation}
and hence we can identify the associative algebras 
\[
\End_\cA(V\otimes\cA)\simeq
\End_\FF(V)\otimes\End_\cA(\cA)\simeq \End_\FF(V)\otimes \cA,
\] 
so that the element $f\otimes a\in\End_\FF(V)\otimes\cA$ is identified with the endomorphism
$f\otimes R_{a^*}\in\End_\cA(V\otimes\cA)$.

Also, for any $f\in \End_\FF(V)$, $u,v\in V$, and $a,b,c\in\cA$, we have:
\[
\hup(f(u)\otimes ba^*,v\otimes c)=2\bigl(f(u)\mid v\bigr)ba^*c^*
=2\bigl(u\mid f^\star(v)\bigr)b(ca)^*,
\]
where $f^\star$ is the adjoint of $f$ relative to $(.\mid.)$: 
$\bigl(f(u)\mid v\bigr)=\bigl(u\mid f^\star(v)\bigr)$ for any $u,v\in V$, so that $f^\star=-f$ if and only 
if $f\in\frsl(V)$. Hence, with the identification
above $\End_\cA(V\otimes\cA)\simeq
\End_\FF(V)\otimes \cA$, the adjoint $(f\otimes a)^\tau$ equals $f^\star\otimes a^*$, 
and}
 the Lie algebra $\Skew\bigl(\End_\cA(V\otimes\cA),\tau\bigr)$ appears as:
\[
\Skew\bigl(\End_\cA(V\otimes\cA),\tau\bigr)\simeq \bigl(\frsl(V)\otimes\cH(\cA,*)\bigr)\oplus 
\bigl(\id_V\otimes\Skew(\cA,*)\bigr),
\]
where $\cH(\cA,*)$ denotes the subspace of symmetric elements of $\cA$ relative to $*$.

Inside the Lie algebra $\Skew\bigl(\End_\cA(\cW),\tau\bigr)$ there is the subalgebra
\[
\cL=\Skew\bigl(\End_\cA(V\otimes\cA),\tau\bigr)\oplus \Skew\bigl(\End_\cA(\cT),\tau\bigr)
\oplus \varphi_{V\otimes\cA,\cT}
\]
which can be identified with
\[
\bigl(\frsl(V)\otimes\cH(\cA,*)\bigr)\oplus \varphi_{V\otimes\cA,\cT}\oplus
\Bigl(\bigl(\id_V\otimes\Skew(\cA,*)\bigr)\oplus \Skew\bigl(\End_\cA(\cT),\tau\bigr)\Bigr).
\]
Also, for any $u\in V$, $a\in\cA$, and $x\in\cT$, we have 
$\varphi_{u\otimes a,x}=\varphi_{u\otimes 1,a^*x}$, so that 
$\varphi_{V\otimes\cA,\cT}=\varphi_{V\otimes 1,\cT}$, and this allows us to identify 
$\varphi_{V\otimes\cA,\cT}$ with $V\otimes\cT$, where $u\otimes x$ corresponds to
$\varphi_{u\otimes 1,x}$ for $u\in V$ and $x\in\cT$. 

With these identifications we may write
\begin{equation}\label{eq:LieHAT}
\cL=\bigl(\frsl(V)\otimes\cH(\cA,*)\bigr)\oplus(V\otimes\cT)\oplus \cD,
\end{equation}
with $\cD=\bigl(\id_V\otimes\Skew(\cA,*)\bigr)\oplus \Skew\bigl(\End_\cA(\cT),\tau\bigr)$. This shows that 
$\cL$ is endowed with a short $\SLs_2$-structure, with $J$-ternary algebra 
$\bigl(\cJ=\cH(\cA,*),\cT\bigr)$.

Let us see what the operations in $(\cJ,\cT)$ look like, according to \eqref{eq:shortA1bracket} and \eqref{eq:shortA1maps}.
\begin{itemize}
\item For $f,g\in\frsl(V)$ and $a,b\in\cH(\cA,*)$, working inside $\End_\FF(V)\otimes\cA$ and 
using \eqref{eq:fggf}, we get:
{\rojo
\[
\begin{split}
[f\otimes a,g\otimes b]&=fg\otimes ab-gf\otimes ba\\
 &=[f,g]\otimes a\cdot b+\frac{1}{2}\trace(fg) \id_V\otimes [a,b],
\end{split}
\]}
and this shows that the Jordan product on $\cJ=\cH(\cA,*)$ is the natural one: 
$a\cdot b=\frac12(ab+ba)$, while $D_{a,b}$ is given by 
$D_{a,b}=\id_V\otimes 4[a,b]\in\id_V\otimes \Skew(\cA,*)$.

\item For $f\in\frsl(V)$, $a\in\cH(\cA,*)$, $v\in V$, and $x\in \cT$, we get:
\[
[f\otimes a,\varphi_{v\otimes 1,x}]=\varphi_{f(v)\otimes a^*,x}=
\varphi_{f(v)\otimes 1,ax},
\]
so the action of $\cJ$ on $\cT$ is also the natural one: $a\bullet x=ax$.

{\rojo
Also, for any $s\in\Skew(\cA,*)$, $a\in\cH(\cA,*)$ and $x\in\cT$, we have
\begin{equation}\label{eq:idVs}
[\id_V\otimes s,\varphi_{v\otimes 1,x}]=\varphi_{v\otimes s^*,x}=\varphi_{v\otimes 1,sx}.
\end{equation}}

\item Finally, for $u,v\in V$ and $x,y\in\cT$, we get
\begin{equation}\label{eq:uxvy}
\begin{split}
[\varphi_{u\otimes 1,x},\varphi_{v\otimes 1,y}]
  &=\varphi_{\varphi_{u\otimes 1,x}(v\otimes 1),y}
          +\varphi_{v\otimes 1,\varphi_{u\otimes 1,x}(y)}\\
  &={\rojo 2(v\mid u)}\varphi_{x,y}+\varphi_{v\otimes 1,u\otimes \hup(y,x)}.
\end{split}
\end{equation}
But for $u,v,w\in V$ and $a\in\cH(\cA,*)$, the map $\varphi_{v\otimes 1,u\otimes a}$ sends
$w\otimes 1$ to
{\rojo
\[
2(w\mid v)u\otimes a+2(w\mid u)v\otimes a^*.
\]
On the other hand, for $a=c+s$ with $c=c^*\in\cH(\cA,*)$ and $s=-s^*\in\Skew(\cA,*)$, using that 
$(u\mid v)w+(v\mid w)u+(w\mid u)v$ is $0$, we get
\[
\begin{split}
2(w\mid v)u&\otimes a+ 2(w\mid u)v\otimes a^*\\
 &=2(w\mid v)u\otimes (c+s)+2(w\mid u)v\otimes (c-s)\\
 &=-2\gamma_{u,v}(w)\otimes c  -2\bigl((v\mid w)u+(w\mid u)v\bigr)\otimes s\\
 &=\gamma_{u,v}(w)\otimes(-a-a^*)+(u\mid v)w\otimes (a-a^*)\\
 &=\Bigl(\gamma_{u,v}\otimes(-a-a^*)-(u\mid v)\id_V\otimes(a-a^*)\Bigr)(w\otimes 1)
 \end{split}
 \]
 (recall that the element $f\otimes a\in\End_\FF(V)\otimes\cA$ is identified with the endomorphism
 $f\otimes R_{a^*}\in \End_\cA(V\otimes\cA)$),
so that we have
\[
\varphi_{v\otimes 1,u\otimes a}=\gamma_{u,v}\otimes (-a-a^*)-(u\mid v)\id_V\otimes (a-a^*),
\]
and then \eqref{eq:uxvy} gives the following identity:
\begin{multline*}
[\varphi_{u\otimes 1,x},\varphi_{v\otimes 1,y}]\\
 =\gamma_{u,v}\otimes \bigl(\hup(x,y)-\hup(y,x)\bigr)
   +(u\mid v)\Bigl(-2\varphi_{x,y}-\bigl(\id_V\otimes \bigl(\hup(x,y)+\hup(y,x)\bigr)\bigr)\Bigr),
\end{multline*}
and \eqref{eq:shortA1bracket} gives the following values:
\begin{equation}\label{eq:xydxy}
\begin{split}
&\langle x\mid y\rangle =\hup(x,y)-\hup(y,x)\,\bigl(=\hup(x,y)+\hup(x,y)^*\in\cH(\cA,*)\bigr),\\
&d_{x,y}=-2\varphi_{x,y}-\id_V\otimes\bigl(\hup(x,y)+\hup(y,x)\bigr)\in 
    \Skew\bigl(\End_\cA(\cT),\tau\bigr)\oplus \bigl(\id_V\otimes \Skew(\cA,*)\bigr).
\end{split}
\end{equation}
Now Equations \eqref{eq:Jternary}, \eqref{eq:phixy}, and \eqref{eq:xydxy}, provide the triple product:
\begin{equation}\label{eq:xyz_proto}
\begin{split}
\ptriple{x,y,z}
     &=\varphi_{x,y}(z)+\frac12 \bigl(\hup(x,y)+\hup(y,x)\bigr)z
         +\frac12\bigl(\hup(x,y)-\hup(y,x)\bigr)z\\
    &=\hup(z,x)y+\hup(z,y)x+\frac12\bigl(\hup(x,y)+\hup(y,x)\bigr)z+\frac12\bigl(\hup(x,y)-\hup(y,x)\bigr)z\\
    &=\hup(x,y)z+\hup(z,x)y+\hup(z,y)x.
\end{split}
\end{equation}}
\end{itemize}

Therefore, our $J$-ternary algebra $(\cJ,\cT)$ is the \emph{prototypical example} of $J$-ternary
algebra in \cite[Example 3.14]{AllisonBenkartGao02}. {\rojo We summarize it in the following result:

\begin{proposition} 
Let $(\cA,*)$ be a unital associative algebra with involution and let   $\cT$ be a left $\cA$-module endowed with a skew-hermitian form 
$\hup:\cT\times\cT\rightarrow \cA$. Then the pair $(\cJ,\cT)$, where $\cJ=\cH(\cA,*)$ is the Jordan
algebra of symmetric elements of $\cA$ relative to the involution $*$, is a $J$-ternary algebra with 
the following operations: 
\begin{itemize}
\item $a\cdot b=\frac{1}{2}(ab+ba)$ for any $a,b\in\cJ=\cH(\cA,*)$,
\item $a\bullet x=ax$ for any $a\in\cJ$ and $x\in\cT$,
\item $\langle x\mid y\rangle = \hup(x,y)-\hup(y,x)$ for $x,y\in\cT$, and
\item $\ptriple{x,y,z}=\hup(x,y)z+\hup(z,x)y+\hup(z,y)x$ for $x,y,z\in\cT$.
\end{itemize}
\end{proposition}}

\bigskip

Let us have a look at some particular cases of this prototypical example. 

To begin with, let $W$ and 
$Z$ be two vector spaces, let $\cB$ be the associative algebra $\End_\FF(W)$ of linear 
endomorphisms of $W$, and let $\cA=\cB\oplus\cB^{op}$ be the direct sum of $\cB$ and its 
opposite algebra (defined over the same vector space but with new multiplication $a.b=ba$). 
Consider the exchange involution on $\cA$: $(a,b)^{ex}=(b,a)$. The Jordan algebra 
$\cJ=\cH(\cA,ex)$ is isomorphic to the Jordan algebra $\cB^{(+)}$ obtained by the symmetrization of
the multiplication in $\cB$.

The vector space $W$ 
is the natural left 
module for $\cB$, and hence it is a left module for $\cA$ annihilated by $\cB^{op}$. The dual vector space $W^*$ is a right $\cB$-module, and hence a left $\cB^{op}$-module, so it becomes a 
left $\cA$-module annihilated by $\cB$. Then $\cT\bydef (W\otimes Z^*)\oplus (W^*\otimes Z)$ 
becomes a left $\cA$-module in a natural way {\rojo (the action of $\cA$ on $Z$ and $Z^*$ is
trivial)}, and it is endowed with a skew-hermitian form
$\hup$ in which the subspaces $W\otimes Z^*$ and $W^*\otimes Z$ are totally isotropic, and
where we have
{\rojo
\[
\hup(w\otimes \alpha, \omega\otimes z)=\alpha(z)w\omega(.)\in\End_\FF(W)=\cB\subseteq\cA
\]
for $w\in W$, $\omega\in W^*$, $z\in Z$, and $\alpha\in Z^*$.}

Then $\bigl(\cJ=\cB^{(+)},\cT\bigr)$ is a $J$-ternary algebra and if $W$ and $Z$ are finite-dimensional, the Lie algebra $\cL$ in \eqref{eq:LieHAT} is naturally isomorphic to the
general linear Lie algebra $\frgl\bigl((V\otimes W)\oplus Z\bigr)$. Its derived algebra
$\frsl\bigl((V\otimes W)\oplus Z\bigr)$ shares the same $J$-ternary algebra.

\smallskip

On the other hand, if $W$ is a vector space endowed with a symmetric 
(respectively skew-symmetric) bilinear form $\bup_W:W\times W\rightarrow \FF$, and 
$Z$ is another vector space endowed with a skew-symmetric (resp. symmetric) bilinear form 
$\bup_Z:Z\times Z\rightarrow \FF$, then $\cT=W\otimes Z$ is naturally a left module for the algebra 
$\cA=\End_\FF(W)$, which is endowed with the involution attached to $\bup_W$, while $\cT$ is 
endowed with the skew-hermitian form
\[
\begin{split}
\hup: \cT\times\cT&\longrightarrow \cA\\
  (w_1\otimes z_1,w_2\otimes z_2)&\mapsto \bup_Z(z_1,z_2)w_1\bup_W(w_2,.).
\end{split}
\]
For finite-dimensional $W$ and $Z$, the Lie algebra
$\cL$ in \eqref{eq:LieHAT} is the symplectic Lie algebra $\frsp\bigl((V\otimes W)\perp Z\bigr)$ (resp. the orthogonal Lie algebra $\frso\bigl((V\otimes W)\perp Z\bigr)$), where the bilinear form on 
$(V\otimes W)\perp Z$ is the orthogonal sum of the form $\bup_Z$ on $Z$ and the tensor product of the skew-symmetric form on $V$ and the form $\bup_W$ on $W$.

\medskip

\subsection{Structurable algebras and \texorpdfstring{$J$}{J}-ternary algebras}\label{ss:Structurable}

Let $\cL$ be a Lie algebra with a short $\SLs_2$-structure and isotypic decomposition as in 
\eqref{eq:shortA1_isotypic}. The Lie bracket is then given by \eqref{eq:shortA1bracket}, and the 
triple product on $\cT$ by \eqref{eq:Jternary}. 

Fix a symplectic basis $\{p,q\}$ in $V$: $(p\mid q)=1$, and identify $\frsl(V)$ with $\frsl_2$ using 
this basis. Then we have
$E=\left(\begin{smallmatrix} 0&1\\ 0&0\end{smallmatrix}\right)=\frac12\gamma_{p,p}$,
$H=\left(\begin{smallmatrix} 1&0\\ 0&-1\end{smallmatrix}\right)=-{\rojo \gamma_{p,q}}$, and
$F=\left(\begin{smallmatrix} 0&0\\ 1&0\end{smallmatrix}\right)=-\frac12\gamma_{q,q}$. 
The adjoint action of $H$ gives a $5$-grading:
\[
\cL=\cL_{-2}\oplus\cL_{-1}\oplus\cL_0\oplus\cL_1\oplus\cL_2,
\]
where $\cL_2=E\otimes \cJ$, $\cL_{-2}=F\otimes\cJ$, $\cL_1=p\otimes\cT$, 
$\cL_{-1}=q\otimes \cT$, and $\cL_0=(H\otimes \cJ)\oplus\cD$.

Identify $\cJ$ with $\cL_2$ by means of $a\leftrightarrow E\otimes a$, and $\cT$ with $\cL_1$ by
means of $x\leftrightarrow p\otimes x$. Writing $F$ for $F\otimes 1$ and using \eqref{eq:shortA1bracket} gives, for $a,b\in \cJ$ 
and $x,y,z\in\cT$, the following equations:
\begin{itemize}
\item 
$[[E\otimes a,F],E\otimes b]=[H\otimes a,E\otimes b]=2E\otimes a\cdot b$, so that the Jordan product in $\cJ$ becomes the product in $\cL_2$  given by
\begin{equation}\label{eq:AB}
A\cdot B=\frac12 [[A,F],B]
\end{equation}
for $A,B\in\cL_2$.

\item
$[[E\otimes a,F],p\otimes x]=[H\otimes a,p\otimes x]=p\otimes a\bullet x$, and hence $\cL_1$ becomes a special module for the Jordan algebra $\cL_2$ with the action given by
\begin{equation}\label{eq:AX}
A\bullet X=[[A,F],X],
\end{equation}
for $A\in\cL_2$ and $X\in\cL_1$.

\item
$[p\otimes x,p\otimes y]=\gamma_{p,p}\otimes\langle x\mid y\rangle
=2E\otimes\langle x\mid y\rangle$, so that the product $\cL_1\times\cL_1\rightarrow\cL_2$ in the
$J$-ternary algebra $(\cL_2,\cL_1)$ is given by
\begin{equation}\label{eq:XY}
\langle X\mid Y\rangle=\frac12 [X,Y],
\end{equation}
for $X,Y\in\cL_1$.

\item
$[[p\otimes x,[{\rojo p}\otimes y,F]],p\otimes z]=-[[p\otimes x,q\otimes y],p\otimes z]
=-[\gamma_{p,q}\otimes\langle x\mid y\rangle+d_{x,y},p\otimes z]=
p\otimes\bigl(\langle x\mid y\rangle\bullet z-d_{x,y}(z)\bigr)=2p\otimes\ptriple{x,y,z}$ (recall
\eqref{eq:Jternary}), and hence
the triple product in $\cL_1$ is given by
\begin{equation}\label{eq:XYZ}
\ptriple{X,Y,Z}=\frac12 [[X,[Y,F]],Z],
\end{equation}
for $X,Y,Z\in\cL_1$.
\end{itemize}

This is summarized in the next result.

\begin{proposition}\label{pr:5graded}
Let $\cL$ be a $5$-graded Lie algebra, and assume that there are elements $E\in\cL_2$ and 
$F\in\cL_{-2}$, such that $\espan{E,F,H=[E,F]}$ is a subalgebra isomorphic to $\frsl_2$, with 
$\cL_i=\{X\in\cL\mid [H,X]=iX\}$, for $i=-2,-1,0,1,2$. {\rojo (In particular $[H,E]=2E$ and 
$[H,F]=-2F$.)} Then the pair $(\cL_2,\cL_1)$ is a 
$J$-ternary algebra with the operations in \eqref{eq:AB}, \eqref{eq:AX}, \eqref{eq:XY}, and
\eqref{eq:XYZ}.
\end{proposition}

\medskip

In particular, consider a structurable algebra $(\cA,\text{-})$, that is, a unital algebra with involution
such that the operators $V_{x,y}$ defined by 
\[
V_{x,y}(z)=\{x,y,z\}\bydef (x\bar y)z+(z\bar y)x-(z\bar x)y
\]
satisfy
\[
[V_{x,y},V_{z,w}]=V_{V_{x,y}(z),w}-V_{z,V_{y,x}(w)}
\]
for all $x,y,z,w\in\cA$ (see \cite{Allison78}). Let $\cH$ and $\cS$ denote, respectively, the subspace of symmetric ($\bar h=h$) and skew-symmetric ($\bar s=-s$) elements in $\cA$. The 
inner structure Lie algebra $\Instrl(\cA,\text{-})$ of $\cA$ is the Lie subalgebra of $\frgl(\cA)$ spanned
by the $V_{x,y}$'s.

The Kantor construction (see \cite{Allison79} or \cite[6.4]{AllisonBenkartGao02}) gives the
$5$-graded Lie algebra
\[
\cL=K(\cA,\text{-})=\cS^\sim\oplus\cA^\sim\oplus\Instrl(\cA,\text{-})\oplus\cA\oplus\cS,
\]
where $\cL_{-2}=\cS^\sim$ is a copy of $\cS$, $\cL_{-1}=\cA^\sim$ is a copy of $\cA$ (the elements in these cases will be written as $s^\sim$ for $s\in\cS$, or $x^\sim$ for $x\in\cA$), 
$\cL_0=\Instrl(\cA,\text{-})$, $\cL_1=\cA$ and $\cL_2=\cS$. The Lie bracket in $\cL$ is given by the following equations:
\[
\begin{aligned}
\null [T,x]&=T(x),      &       [T,x^\sim]&=T^\epsilon(x)^\sim,\\
[T,s]&=T(s)+s\overline{T(1)},\qquad\null   &   
   [T,s^\sim]&=\bigl(T^\epsilon(s)+s\overline{T^\epsilon(1)}\bigr)^\sim,\\
[x,y]&=2(x\bar y-y\bar x),  &  [x^\sim,y^\sim]&=2(x\bar y-y\bar x)^\sim,\\
[x,y^\sim]&=2V_{x,y}, &  [s,t^\sim]&=L_sL_t\Bigl(=\frac12 (V_{st,1}-V_{s,t})\Bigr),\\
[x,s^\sim]&=-(sx)^\sim, &  [x^\sim,s]&=-sx,
\end{aligned}
\]
for any $x,y\in\cA$, $s,t\in\cS$, and $T\in\Instrl(\cA,\text{-})$, and where $L_s$ is the operator of left multiplication by $s$, and $V_{x,y}^\epsilon\bydef -V_{y,x}$ for any $x,y\in\cA$.

\begin{proposition}\label{pr:structurable}
Let $(\cA,\text{-})$ be a structurable algebra, and let $s\in\cS$ be a skew-symmetric element such
that the left multiplication $L_s$ is bijective, then the pair $(\cS,\cA)$ is a $J$-ternary algebra with 
the following operations:
\[
\begin{split}
a\cdot b&=\frac12\bigl(a(sb)+b(sa)\bigr),\\
a\bullet x&=a(sx),\\
\langle x\mid y\rangle&=x\bar y-y\bar x,\\
\ptriple{x,y,z}&=-V_{x,sy}(z),
\end{split}
\]
for $a,b\in\cS$ and $x,y,z\in \cA$.
\end{proposition}
\begin{proof}
In a slightly different way, this appears without proof in \cite[Remark 6.7]{AllisonBenkartGao02}. 
We will prove it here as a consequence of Proposition \ref{pr:5graded}. 

Since $L_s$ is invertible, there is an element $s'\in\cS$ such that $L_s^{-1}=L_{s'}$ 
(\cite[Proposition 11.1]{AllisonHein}) and then the elements $E=s'$ and $F=s^\sim$ satisfy the
conditions of Proposition \ref{pr:5graded}, with $H=[E,F]=V_{1,1}=\id$. Therefore, $(\cS,\cA)$
is a $J$-ternary algebra with the operations defined as in \eqref{eq:AB}, \eqref{eq:AX}, \eqref{eq:XY}, and \eqref{eq:XYZ}:
\[
\begin{split}
a\cdot b&=\frac12 [[a,s^\sim],b]=\frac12[L_aL_s,b]
   =\frac12 \bigl(a(sb)+b(\overline{as})\bigr)=\frac 12\bigl(a(sb)+b(sa)\bigr),\\
a\bullet x&=[[a,s^\sim],x]=[L_aL_s,x]=a(sx),\\
\langle x\mid y\rangle&=\frac12 [x,y]=x\bar y-y\bar x,\\
\ptriple{x,y,z}&=\frac12 [[x,[y,s^\sim]],z]=\frac12 [[x,-(sy)^\sim],z]
     =-[V_{x,sy},z]=-V_{x,sy}(z),
\end{split}
\]
for $a,b\in\cS$ and $x,y,z\in\cA$, as required.
\end{proof}

As a particular case, consider the case of the {\rojo structurable algebra obtained as the} tensor product of a Cayley (or octonion) algebra and another unital composition algebra: $\cA=\cC_1\otimes \cC_2$, where the involution is the tensor product of the
canonical involutions in the Cayley algebra $\cC_1$ and in the composition algebra $\cC_2$. 

Denote by $\nup_i$ the norm in $\cC_i$, and by $\cS_i$ the subspace of trace zero elements (i.e., 
$\bar s=-s$) in $\cC_i$. The skew-symmetric part of $\cA$ is 
$\cS=\cS_1\otimes 1 + 1\otimes\cS_2$, which we will identify with $\cS_1\oplus\cS_2$. Fix an element $s\in\cS_1$ with $\nup_1(s)\neq 0$.

As in \cite{Allison88} consider the \emph{Albert form} $Q:\cS\rightarrow\FF$, which is the 
nondegenerate quadratic form given by 
\[
Q(s_1+s_2)\bydef \frac{1}{\nup_1(s)}\bigl(\nup_1(s_1)-\nup_2(s_2)\bigr),
\]
for $s_1\in\cS_1$ and $s_2\in\cS_2$. Consider also the linear map $\sharp$ given by
\[
(s_1+s_2)^\sharp \bydef \nup_1(s)(s_1-s_2).
\]
Write $c=-\frac{1}{\nup_1(s)}s$. The pair $(\cS,\cA)$ is a $J$-ternary algebra 
(Proposition \ref{pr:structurable}), and the Jordan product in $\cS$ satisfies, for any $a\in\cS$,
 the following:
\begin{equation}\label{eq:Scdot}
a^{\cdot 2}=asa=\frac{1}{\nup_1(s)}(as^\sharp a)=
\frac{1}{\nup_1(s)}\bigl(Q(a)s-Q(a,s)a\bigr)=Q(a,c)c-Q(a)c,
\end{equation}
where we have used \cite[(3.7)]{Allison88}. Therefore $\cS$ is the Jordan algebra denoted
by $\cJord(Q,c)$ in \cite[I.3.7]{McCrimmon} (a \emph{quadratic factor} in the notation there), or
the \emph{Jordan algebra of the quadratic form} $-Q\vert_{(\FF s)^\perp}$ in the notation
of \cite[I.4]{JacobsonJordan}. The unity is the element $c$. (Here $(\FF s)^\perp$ denotes the
subspace of $\cS$ orthogonal to $s$.)

Our structurable algebra $\cA=\cC_1\otimes\cC_2$ is a special module for the Jordan algebra
$(\cS,\cdot)$ with $a\bullet x=a(sx)$. For any $a\in\cS$ orthogonal to $s$ (relative to $Q$), and 
any $x\in \cA$ we get
\[
a\bullet(a\bullet x)=a^{\cdot 2}\bullet x=-Q(a)c\bullet x=-Q(a)x,
\]
and hence the action of $\cS$ on $\cA$ extends to a homomorphism of associative algebras.
\begin{equation}\label{eq:theta}
\theta:\Cl\bigl((\FF s)^\perp,-Q\vert_{(\FF s)^\perp}\bigr)\rightarrow \End_\FF(\cA),
\end{equation}
where $\Cl\bigl((\FF s)^\perp,-Q\vert_{(\FF s)^\perp}\bigr)$ is the associated Clifford algebra. 
This is the unital special universal envelope of the Jordan algebra $\cS$ 
\cite[II.11]{JacobsonJordan}. Actually, the Jordan algebra $(\cS,\cdot)$ embeds in
$\Cl\bigl((\FF s)^\perp,-Q\vert_{(\FF s)^\perp}\bigr)^{(+)}$, with $\alpha c+u\in\cS$ being sent to
$\alpha 1+u\in \Cl\bigl((\FF s)^\perp,-Q\vert_{(\FF s)^\perp}\bigr)$, for any $\alpha\in \FF$ and 
$u\in(\FF s)^\perp$ (recall that $c$ is the unity of the Jordan algebra $(\cS,\cdot)$).

\begin{remark}\label{re:Clifford}
The Clifford algebra $\Cl\bigl((\FF s)^\perp,-Q\vert_{(\FF s)^\perp}\bigr)$ is isomorphic to the even Clifford algebra $\Cl^+\bigl(\cS,Q\bigr)$. The natural isomorphism takes any $a\in(\FF s)^\perp$
to the element $\frac{1}{\nup_1(s)}a\diamond s$, where $\diamond$ denotes the (associative) 
product in $\Cl\bigl(\cS,Q\bigr)$.
\end{remark}

\medskip

If $\cC_2$ is associative, and hence of dimension at most $4$, then $\cA$ is a free right 
$\cC_2$-module of dimension $8$ and the image of $\theta$ in \eqref{eq:theta} is contained in 
$\End_{\cC_2}(\cA)\simeq\End_\FF(\cC_1)\otimes\cC_2$. Actually, \cite[Theorem 4.5]{Allison88}
shows that the image of $\theta$ is precisely $\End_{\cC_2}(\cA)$. Moreover, if $\cC_2$ is commutative (so its dimension is $1$ or $2$), by dimension count, $\theta$ gives an isomorphism
$\Cl\bigl((\FF s)^\perp,-Q\vert_{(\FF s)^\perp}\bigr)\simeq \End_{\cC_2}(\cA)$. If $\cC_2$ is a quaternion algebra, then $\Cl\bigl((\FF s)^\perp,-Q\vert_{(\FF s)^\perp}\bigr)$ has dimension
$2^9$, while $\End_{\cC_2}(\cA)\simeq \Mat_8(\cC_2)$ has dimension $4\times 8^2=2^8$. In this case $\Cl\bigl((\FF s)^\perp,-Q\vert_{(\FF s)^\perp}\bigr)$ is isomorphic to the tensor product of
the even Clifford algebra, which is simple, and the two-dimensional center. The existence of $\theta$
shows that $\Cl\bigl((\FF s)^\perp,-Q\vert_{(\FF s)^\perp}\bigr)$ is not simple, and hence its center
is isomorphic to $\FF\times\FF$. It follows that $\Cl\bigl((\FF s)^\perp,-Q\vert_{(\FF s)^\perp}\bigr)$
is isomorphic to $\Mat_8(\cC_2)\times\Mat_8(\cC_2)$.

If $\cC_2$ is not associative, then it is an eight-dimensional Cayley algebra, and the dimension
of $\Cl\bigl((\FF s)^\perp,-Q\vert_{(\FF s)^\perp}\bigr)$ is $2^{13}$, while the dimension of
$\End_\FF(\cA)$ is $64^2=2^{12}$. As in the previous case, we conclude that
$\Cl\bigl((\FF s)^\perp,-Q\vert_{(\FF s)^\perp}\bigr)$ is isomorphic to 
$\Mat_{64}(\FF)\times\Mat_{64}(\FF)$.

\smallskip
{\rojo
\begin{remark} In a different setting, the $J$-ternary algebras obtained from the structurable algebras
$\cC_1\otimes\cC_2$ have been considered  too in \cite{BdM}.
\end{remark}}

\bigskip

\section{Short \texorpdfstring{$(\SLs_2\times\SLs_2)$}{SL2xSL2}-structures on 
finite-dimensional simple Lie algebras}\label{se:A1A1simple}

The aim of this section is to take advantage of the description of the short $\SLs_2$-structures on the 
finite-dimensional simple Lie algebras in \cite{Stasenko}, obtained by a careful analysis of the
$5$-gradings on these algebras, together with Theorem \ref{th:A1A1e}, to describe the
short $(\SLs_2\times\SLs_2)$-structures on these algebras. {\rojo All these structures are obtained
from the examples considered in the previous section.}

Throughout the section, the ground field $\FF$ will be assumed to be algebraically closed and of
characteristic $0$.

\medskip

\subsection{Special linear Lie algebras}\label{ss:special}

The results in \cite[\S 4.2]{Stasenko} show that for a finite-dimensional simple Lie algebra of type $A$
(special linear), the only $\SLs_2$-structures are given by viewing it as 
$\frsl\bigl((V\otimes W)\oplus Z\bigr)$, where $V$ is our two-dimensional vector space endowed with
a nonzero alternating bilinear form $(.\mid.)$, and $W$ and $Z$ are two other vector spaces.
In this case, the associated $J$-ternary algebra $(\cJ,\cT)$ is given in Subsection \ref{ss:proto} 
(prototypical example), with $\cJ=\End_\FF(W)^{(+)}$ and $\cT=(W\otimes Z^*)\oplus(W^*\otimes Z)$.

The proper idempotents $e\neq 0,1$ of $\End_\FF(W)^{(+)}$ exist only if $\dim_\FF W>1$, and they 
are just the projections relative to a splitting $W=W_1\oplus W_0$: $e(w_1)=w_1$ and $e(w_0)=0$, 
for $w_1\in W_1$ and $w_0\in W_0$. For any such idempotent there is an associated 
short $(\SLs_2\times\SLs_2)$-structure, and the components $\cJ_1$, $\cJ_0$, $\cJ_{\frac12}$ in
\eqref{eq:A1A1isotypic} are the Peirce components relative to $e$: $\cJ_1=\End_\FF(W_1)^{(+)}$,
$\cJ_0=\End_\FF(W_0)^{(+)}$, and $\cJ_{\frac12}=\Hom_\FF(W_1,W_0)\oplus\Hom_\FF(W_0,W_1)$; 
while $\cT_i=(W_i\otimes Z^*)\oplus (W_i^*\otimes Z)$, for $i=1,0$.

\medskip

\subsection{Orthogonal Lie algebras}\label{ss:orthogonal}

Also, the results in \cite[\S 4.2]{Stasenko} show that  for a finite-dimensional simple 
Lie algebra of type $B$ or $D$
(orthogonal), the only $\SLs_2$-structures are given by viewing it as 
$\frso\bigl((V\otimes W)\perp Z\bigr)$, with $V$ as above, $W$ endowed with a nondegenerate
skew-symmetric bilinear form $\bup_W$, and $Z$ endowed with a nondegenerate symmetric
bilinear form $\bup_Z$. The symmetric bilinear form on $(V\otimes W)\perp Z$ is the orthogonal 
sum of $(.\mid .)\otimes \bup_W$ and $\bup_Z$. In this case, Subsection \ref{ss:proto} shows that
the associated $J$-ternary algebra $(\cJ,\cT)$ is given by $\cJ=\cH\bigl(\End_\FF(W),\tau\bigr)$, where
$\tau$ is the symplectic involution relative to $\bup_W$, and $\cT=W\otimes Z$.

The Jordan algebra $\cJ=\cH\bigl(\End_\FF(W)^{(+)},\tau\bigr)$ contains proper idempotents if and only
if the dimension of $W$ is at least $4$ (note that this dimension is always even). Any proper idempotent
is the projection relative to an orthogonal decomposition $W=W_1\perp W_0$ relative to $\bup_W$.
For any such idempotent, the components $\cJ_1$, $\cJ_0$, $\cJ_{\frac12}$ in
\eqref{eq:A1A1isotypic} are the Peirce components relative to $e$:  
$\cJ_1=\cH\bigl(\End_\FF(W_1)^{(+)},\tau\bigr)$, 
$\cJ_0=\cH\bigl(\End_\FF(W_0)^{(+)},\tau\bigr)$, and 
$\cJ_{\frac12}=\{ f\in \cH\bigl(\End_\FF(W)^{(+)},\tau\bigr)\mid f(W_1)\subseteq W_0,\, 
f(W_0)\subseteq W_1\}$; while $\cT_i=W_i\otimes Z$, $i=1,0$.

\medskip

\subsection{Symplectic Lie algebras}\label{ss:symplectic}

Finally, \cite[\S 4.2]{Stasenko} shows that  for a finite-dimensional simple Lie algebra of type $C$
(symplectic), the only $\SLs_2$-structures are given by viewing it as 
$\frsp\bigl((V\otimes W)\perp Z\bigr)$, with $V$ as above, $W$ endowed with a nondegenerate
symmetric bilinear form $\bup_W$, and $Z$ endowed with a nondegenerate skew-symmetric
bilinear form $\bup_Z$. The skew-symmetric bilinear form on $(V\otimes W)\perp Z$ is the orthogonal 
sum of $(.\mid .)\otimes \bup_W$ and $\bup_Z$. In this case, Subsection \ref{ss:proto} shows that
again the associated $J$-ternary algebra $(\cJ,\cT)$ is given by $\cJ=\cH\bigl(\End_\FF(W),\tau\bigr)$, 
where $\tau$ is the orthogonal involution relative to $\bup_W$, and $\cT=W\otimes Z$.

The Jordan algebra $\cJ=\cH\bigl(\End_\FF(W)^{(+)},\tau\bigr)$ contains proper idempotents if and only
if $\dim_\FF W>1$. Any proper idempotent
is again the projection relative to an orthogonal decomposition $W=W_1\perp W_0$ relative to $\bup_W$.
For any such idempotent, the components $\cJ_1$, $\cJ_0$, $\cJ_{\frac12}$ in
\eqref{eq:A1A1isotypic} are the Peirce components relative to $e$:  
$\cJ_1=\cH\bigl(\End_\FF(W_1)^{(+)},\tau\bigr)$, 
$\cJ_0=\cH\bigl(\End_\FF(W_0)^{(+)},\tau\bigr)$, and 
$\cJ_{\frac12}=\{ f\in \cH\bigl(\End_\FF(W)^{(+)},\tau\bigr)\mid f(W_1)\subseteq W_0,\, 
f(W_0)\subseteq W_1\}$; while $\cT_i=W_i\otimes Z$, $i=1,0$.

\medskip

\subsection{Exceptional Lie algebras}\label{ss:exceptional}

The results in \cite[\S 4.3]{Stasenko} show that for any of the finite-dimensional exceptional simple Lie 
algebras there is a short $\SLs_2$-structure whose associated Jordan algebra $\cJ$ is just the ground
field $\FF$. Then $\cT$, endowed with the skew-symmetric bilinear form $\langle .\mid .\rangle$ and the 
triple product $[x,y,z]\bydef d_{x,y}(z)$ is a \emph{symplectic triple system} 
(see \cite{YamagutiAsano}).
Alternatively {\rojo (see \cite[Theorem 4.7]{EldMagicII})}, the triple product 
$\{x,y,z\}\bydef d_{x,y}(z)-\langle x\mid z\rangle y-\langle y\mid z\rangle x$, is either trivial or endows
$\cT$ with the structure of a \emph{Freudenthal triple system} (see \cite{Meyberg}). The classification
of the finite-dimensional simple symplectic triple systems may be consulted in 
\cite[Theorem 2.21]{EldNew}. With one exception, they are obtained from the simple structurable algebras whose subspace of skew-symmetric elements has dimension one.

Apart from the short $\SLs_2$-structures above, where the Jordan algebra $\cJ$ has dimension $1$ and
hence has no proper idempotent, \cite[\S 4.3]{Stasenko} shows that there is no other short 
$\SLs_2$-structure for the simple Lie algebra of type $G_2$, and exactly one other short 
$\SLs_2$-structure in the remaining cases: $F_4$, $E_6$, $E_7$, and $E_8$. Therefore, Subsection
\ref{ss:Structurable} shows that the associated $J$-ternary algebra $(\cJ,\cT)$ can be obtained
from a structurable algebra $\cA=\cC_1\otimes\cC_2$, for a Cayley algebra $\cC_1$ and a unital
composition algebra $\cC_2$ of dimension $1$ for case $F_4$, dimension $2$ for $E_6$, dimension
$4$ for $E_7$, and dimension $8$ for $E_8$. In all these cases, the Jordan algebra $\cJ$ is the
Jordan algebra of a quadratic form, and there is a unique orbit, under its automorphism group, of
proper idempotents. As above, for any such  idempotent, the components 
$\cJ_1$, $\cJ_0$, $\cJ_{\frac12}$ in
\eqref{eq:A1A1isotypic} are the Peirce components relative to $e$:  
$\cJ_1=\FF e$, 
$\cJ_0=\FF(1-e)$ (recall that the unity $1$ is the element $c$ in \eqref{eq:Scdot}), and 
$\cJ_{\frac12}=\bigl(\FF e+\FF(1-e)\bigr)^\perp$; 
while $\cT_i=\{x\in\cT\mid e\bullet x=ix\}$, $i=1,0$.

\bigskip

\section*{Acknowledgments}
The authors are indebted to an anonymous referee for the careful reading of the manuscript and the suggestions offered to improve it.

%
%
%



\bigskip

\begin{thebibliography}{McM04}

\bibitem[All76]{Allison}
B.N.~Allison, \emph{A construction of Lie algebras from $J$-ternary algebras}, 
Amer. J. Math. \textbf{98} (1976), no. 2, 285-294.

\bibitem[All78]{Allison78}
B.N.~Allison, \emph{A class of nonassociative algebras with involution containing the class of Jordan algebras}, Math. Ann. \textbf{237} (1978), no.~2, 133-156.

\bibitem[All79]{Allison79}
B.N.~Allison, \emph{Models of isotropic simple Lie algebras},
Comm. Algebra \textbf{7} (1979), no.~17, 1835-1875.

\bibitem[All88]{Allison88}
B.N.~Allison, \emph{Tensor products of composition algebras, Albert forms and some exceptional simple Lie algebras},
Trans. Amer. Math. Soc. \textbf{306} (1988), no.~2, 667-695.


\bibitem[AH81]{AllisonHein}
B.N.~Allison and W.~Hein, \emph{Isotopes of some nonassociative algebras with involution},
J.~Algebra \textbf{69} (1981), no.~1, 120-142.

\bibitem[ABG02]{AllisonBenkartGao02}
B.N.~Allison, G.~Benkart, and Y.~Gao, \emph{Lie Algebras Graded by the Root System $BC_r$, $r\geq 2$}, Mem. Amer. Math. Soc. \textbf{158} (2002), no. 751, x+158 pp.

\bibitem[BE12]{BenkartElduque_sl3}
G.~Benkart and A.~Elduque, \emph{Lie algebras with prescribed $\frsl_3$-decomposition},
Proc. Amer. Math. Soc. \textbf{140} (2012), no.~8, 2627-2638.

\bibitem[BO82]{BO}
G.M.~Benkart and J.M.~Osborn,
\emph{Representations of rank one Lie algebras of characteristic $p$},
 in ``Lie algebras and related topics'' (New Brunswick, N.J., 1981), pp. 1--37, 
Lecture Notes in Math. \textbf{933}, Springer, Berlin-New York, 1982.

\bibitem[BS03]{BenkartSmirnov}
G.~Benkart and O.~Smirnov, \emph{Lie algebras graded by the root system $BC_1$},
J.~Lie Theory \textbf{13} (2003), no.~1, 91-132.

\bibitem[BM92]{BermanMoody}
S.~Berman and R.V.~Moody, \emph{Lie algebras graded by finite root systems and the
intersection matrix algebras of Slodowy}, Invent. Math. \textbf{108} (1992), 323-347.

{\rojo
\bibitem[BdM15]{BdM}
L.~Boelaert, and T.~De Medts,
\emph{A new construction of Moufang quadrangles of type $E_6$, $E_7$ and $E_8$},
Trans. Amer. Math. Soc. \textbf{367} (2015), no.~5, 3447-3480.}

\bibitem[CE20]{CunhaElduque}
I.~Cunha and A.~Elduque, \emph{Codes, S-structures, and exceptional Lie algebras}, 
Glasg. Math. J. \textbf{62} (2020), no. S1, S14-S27.

\bibitem[Eld06]{EldNew}
A.~Elduque, 
\emph{New simple Lie superalgebras in characteristic $3$},
J.~Algebra \textbf{296} (2006), no.~1, 196-233.

\bibitem[Eld07]{EldMagicII}
A.~Elduque, 
\emph{The magic square and symmetric compositions II},
Rev.~Mat.~Iberoamericana \textbf{23} (2007), no.~1, 57-84.

\bibitem[EK13]{EKmon} 
A.~Elduque and M.~Kochetov, \emph{Gradings on simple Lie algebras}, Mathematical Surveys and Monographs, vol.~189,  American Mathematical Society, Providence, RI, 2013.

\bibitem[EO07]{EOS4}
A.~Elduque and S.~Okubo, 
\emph{Lie algebras with $S_4$-action and structurable algebras} 
J.~Algebra \textbf{307} (2007), no.~2, 864-890. 

\bibitem[EO09]{EOS3S4}
A.~Elduque and S.~Okubo, 
\emph{Lie algebras with $S_3$- or $S_4$-action and generalized Malcev algebras}, Proc. Roy. Soc. Edinburgh Sect.~A \textbf{139} (2009), no.~2, 321-357.

\bibitem[EO11]{EO}
A.~Elduque and S.~Okubo, 
\emph{Special Freudenthal-Kantor triple systems and Lie algebras
 with dicyclic symmetry} 
Proc. Roy. Soc. Edinburgh Sect. A \textbf{141} (2011), no. 6, 1225--1262.

\bibitem[Hei81]{Hein1}
W.~Hein, \emph{On the structure of reduced $\cJ$-ternary algebras of degree $>2$}, 
Math.~Z. \textbf{176} (1981), no.~4, 521-539.

\bibitem[Hei83]{Hein2}
W.~Hein, \emph{On the structure of reduced $\cJ$-ternary algebras of degree two}, 
J.~Algebra \textbf{82} (1983), no.~1, 157-184.

\bibitem[Jac68]{JacobsonJordan}
N.~Jacobson, \emph{Structure and representations of Jordan algebras}, 
Amer. Math. Soc. Colloq. Publ., vol. \textbf{39}, Amer. Math. Soc., Providence, R. I., 1968.

\bibitem[Jac79]{Jacobson}
N.~Jacobson, \emph{Lie algebras}, 
republication of the 1962 original. Dover Publications, Inc., New York, 1979.

\bibitem[McM04]{McCrimmon}
K.~McCrimmon,
\emph{A taste of Jordan algebras},
Universitext. Springer-Verlag, New York, 2004.

\bibitem[Mey68]{Meyberg}
K. Meyberg, 
\emph{Eine Theorie der Freudenthalschen Tripelsysteme. I, II}, Nederl. Akad. Wetensch. Proc. Ser. A 
\textbf{71} (1968) 162-174, 175-190 =Indag. Math. \textbf{30} (1968).

\bibitem[Sta22]{Stasenko}
R.O.~Stasenko, \emph{Short $SL_2$-structures on simple Lie algebras},
preprint (2022),\newline arXiv:2206.13735.

\bibitem[Tits62]{Tits}
J.~Tits, \emph{Une classe d'alg\`ebres de Lie en relation avec les alg\`ebres de Jordan}, 
Indag. Math. \textbf{24} (1962), 530-535.

\bibitem[Vin17]{Vinberg}
E.B.~Vinberg, \emph{Non-abelian gradings of Lie algebras}, 50th Seminar ``Sophus Lie'', 19--38, Banach Center Publ., \textbf{113}, Polish Acad. Sci. Inst. Math., Warsaw, 2017.

\bibitem[YA75]{YamagutiAsano}
K.~Yamaguti and H.~Asano, 
\emph{On the Freudenthal's construction of exceptional Lie algebras}, 
Proc. Japan Acad. \textbf{51} (1975), no.~4, 253-258.



\end{thebibliography}
\end{document}